\documentclass[12pt]{amsart}
\usepackage{tikz-cd}
\usetikzlibrary{matrix,arrows,decorations.pathmorphing}
\usepackage{amssymb}
\usepackage{amsmath,amscd}
\usepackage{mathrsfs}
\usepackage{url}
\usepackage{cite}
\usepackage{fullpage}
\usepackage{hyperref}
\usepackage{verbatim}
\usepackage{comment}
\usepackage{fancyvrb}
\usepackage{fvextra}
\newtheorem{thm}{Theorem}[section]

\newtheorem{lem}[thm]{Lemma}
\newtheorem{cor}[thm]{Corollary}

\theoremstyle{definition}
\newtheorem{definition}[thm]{Definition}

\newtheorem{rem}[thm]{Remark}

\numberwithin{equation}{section}

\renewcommand{\L}{\mathcal{L}}

\newcommand{\T}{\mathcal{T}}
\newcommand{\V}{\mathcal{V}}

\newcommand{\N}{\mathcal{N}}
\newcommand{\B}{\mathcal{B}}

\newcommand{\qq}{\mathbb{Q}}

\newcommand{\C}{\mathcal{C}}
\newcommand{\M}{\mathcal{M}}

\newcommand{\pp}{\mathbb{P}}

\renewcommand{\H}{\mathcal{H}}
\newcommand{\F}{\mathcal{F}}
\renewcommand{\P}{\mathcal{P}}
\newcommand{\E}{\mathcal{E}}
\renewcommand{\O}{\mathcal{O}}

\newcommand{\Mg}{\mathcal{M}_g}

\DeclareMathOperator{\rank}{rank}

\DeclareMathOperator{\ch}{ch}
\DeclareMathOperator{\td}{td}

\DeclareMathOperator{\Pic}{Pic}

\DeclareMathOperator{\coker}{coker}

\DeclareMathOperator{\Sym}{Sym}

\DeclareMathOperator{\BSL}{BSL}

\newcommand{\Mb}{\overline{\mathcal{M}}}

\begin{document}
\title{The bielliptic locus in genus $11$}

\author{Samir Canning and Hannah Larson}
\thanks{During the preparation of this article, S.C. was partially supported by NSF RTG grant DMS-1502651. H.L. was partially supported by the Hertz Foundation and NSF GRFP under grant DGE-1656518. This research was partially conducted during the period H.L served as a Clay Research Fellow.}
\email{scanning@math.ethz.ch}
\email{hlarson@math.harvard.edu}
\maketitle

\begin{abstract}
The Chow ring of $\M_g$ is known to be generated by tautological classes for $g \leq 9$.
Meanwhile, the first example of a non-tautological class on $\M_{g}$ is the fundamental class of the bielliptic locus in $\M_{12}$, due to van Zelm.
It remains open if the Chow rings of $\M_{10}$ and $\M_{11}$ are generated by tautological classes.
In these cases, a natural first place to look is at the bielliptic locus.
In genus $10$, it is already known that classes supported on the bielliptic locus are tautological.
Here, we
prove that all classes supported on the bielliptic locus are tautological in genus $11$. By Looijenga's vanishing theorem, this implies that they all vanish.
\end{abstract}


\section{Introduction}
The Chow rings of the moduli spaces of smooth and stable curves, possibly with marked points, have been central objects of study in algebraic geometry since Mumford's seminal paper \cite{Mumford}. A complete description of the Chow ring of the moduli space $\Mg$ of smooth curves of genus $g$ has been obtained for $g \leq 9$ in a series of papers over the past 40 years:
\begin{itemize}
    \item Genus $2$ by Mumford in 1983 \cite{Mumford},
    \item Genus $3$ and $4$ by Faber in 1990 \cite{FaberI,FaberII},
    \item Genus $5$ by Izadi in 1995 \cite{Izadi},
   \item Genus $6$ by Penev and Vakil in 2015 \cite{PenevVakil},
    \item Genus $7, 8, 9$ by the authors in 2021 \cite{789}.
\end{itemize}
In each of these cases, the Chow ring is shown to be equal to the \emph{tautological ring}.
If $f: \C \to \M_g$ denotes the universal curve, then the tautological subring $R^*(\M_g) \subseteq A^*(\M_g)$ is defined to be the subring generated by the classes $\kappa_i := f_*(c_1(\omega_f)^{i+1})$.

Even more recently, there has been significant progress in the study of the Chow rings of the moduli spaces $\Mb_{g,n}$ of stable curves of genus $g$ with $n$ marked points \cite{CL-CKgP,CanningLarsonPayne}. The Chow ring $A^*(\Mb_{g,n})$ and its tautological subring $R^*(\Mb_{g,n})$ are of fundamental importance in a wide variety of fields, including enumerative geometry and physics \cite{Kontsevich, Witten}. The first step in the study of $A^*(\Mb_{g,n})$ is to compute $A^*(\M_{g,n})$. Therefore, $A^*(\M_{g,n})$ is an important object of study both in its own right and for its applications to computing $A^*(\Mb_{g,n})$.

The tautological ring $R^*(\M_g)$ is highly structured. By a theorem of Looijenga, $R^i(\M_g)$ vanishes when $i>g-2$ and by combining results of Faber and Looijenga, $R^{g-2}(\M_g)=\qq$, generated by the class of the hyperelliptic locus \cite{Looijenga,Fabernonvanishing}.
Looijenga's vanishing theorem can be used to obtain interesting vanishing results for classes of codimension at least $g-1$ by demonstrating that such a class is tautological.

In general, however, tautological classes do not generate the entire Chow ring. In fact, van Zelm  \cite{VZ} has produced an explicit algebraic cycle on $\M_{12}$ which is \emph{not} tautological. 
A curve $C$ is called \emph{bielliptic} if it admits a degree $2$ map $C \to E$ for $E$ an elliptic curve. 
Let $\B_g \subset \M_g$ denote the locus of bielliptic curves. A dimension count using the Riemann--Hurwitz formula shows that $\B_g \subset \M_g$ has codimension $g-1$. 
In \cite{VZ},
van Zelm proves that $[\B_{12}]$ is \emph{not} tautological. Furthermore, for all $g \geq 12$, van Zelm shows the closure of the bielliptic locus $\overline{\B}_{g} \subset \Mb_g$ has non-tautological fundamental class.

In light of van Zelm's results, it remains an open problem of particular interest whether the Chow rings of $\M_{10}$ and $\M_{11}$ are generated by tautological classes. To approach this problem, one might try stratifying the moduli space by gonality, as we did in \cite{789}. However, difficulties arise both with large and small gonality. In large gonality, we do not have good control over the Hurwitz space parametrizing covers of degree at least $6$, which are general in $\Mg$ for $g\geq 9$. In small gonality, we do not know if certain fundamental classes in a further stratification are tautological when $g\geq 10$. See Section \ref{sts} for a discussion of this issue.

A natural first place to check for non-tautological classes is the bielliptic locus, following van Zelm. In \cite{789}, we proved that $[\B_{10}]$ is tautological. The argument there relies on a special coincidence of small numbers that holds only in genus $10$, allowing us to bypass the aforementioned issues with fundamental classes in small gonality. Here, we settle the remaining case of genus $11$, using a new technique.

\begin{thm} \label{main}
The fundamental class of the bielliptic locus $[\B_{11}]$ is tautological. As a consequence, all classes supported on $\B_{11} \subset \M_{11}$ are tautological.
\end{thm}
\noindent Combining this result with Looijenga's vanishing theorem \cite{Looijenga}, we see that all classes supported on $\B_{11} \subset \M_{11}$ are in fact zero for codimension reasons.

Theorem \ref{main} provides the first positive evidence that Chow ring of $\M_{11}$ may be generated by tautological classes.
The proof demonstrates an important concept that we hope may find other future applications: even when descriptions of a space as a degeneracy locus fail to occur in the correct codimension, excess intersection formulas may still provide enough information to deduce a desired class is tautological.

\subsection{Overview of the proof}\label{overview}
Our proof will utilize the Hurwitz space $\H_{4,g}$, parametrizing degree $4$, genus $g$ covers of $\pp^1$ up to automorphisms of the target. We write $\beta: \H_{4,g} \to \M_g$ for the forgetful map. The basic idea is to find a tautological class on $\H_{4,11}$ that pushes forward along $\beta$ to a \emph{non-zero} multiple of $[\B_{11}]$. A little care is needed to make this precise because the map $\beta$ is not proper, since our degree $4$ map could acquire a base point. However, $\beta$ becomes proper after throwing out curves of lower gonality from the source and target.

More precisely, let $\M_g^3 \subset \M_g$ denote the locus of curves of gonality at most $3$. Then
\[\beta': \H_{4,g} \smallsetminus \beta^{-1}(\M_g^3) \to \M_g \smallsetminus \M_g^3\]
is proper, and $\beta'_*$ sends tautological classes on $\H_{4,g} \smallsetminus \beta^{-1}(\M_g^3)$ to tautological classes on $\M_g \smallsetminus \M_g^3$ by \cite[Theorem 1.7]{part2}. Here, the tautological subring of an open subset of $\H_{4,g}$ or $\M_g$ is defined to be the image of the tautological ring under restriction. Fortunately, it is already known that all classes supported on $\M_g^3$ are tautological \cite[Section 1.1(1)]{789}. Thus, by excision, it will suffice to show that the image of $[\B_{11}]$ in $\M_{11} \smallsetminus \M_{11}^3$ is tautological.

\begin{figure}
    \centering
    \includegraphics[width=6.3in]{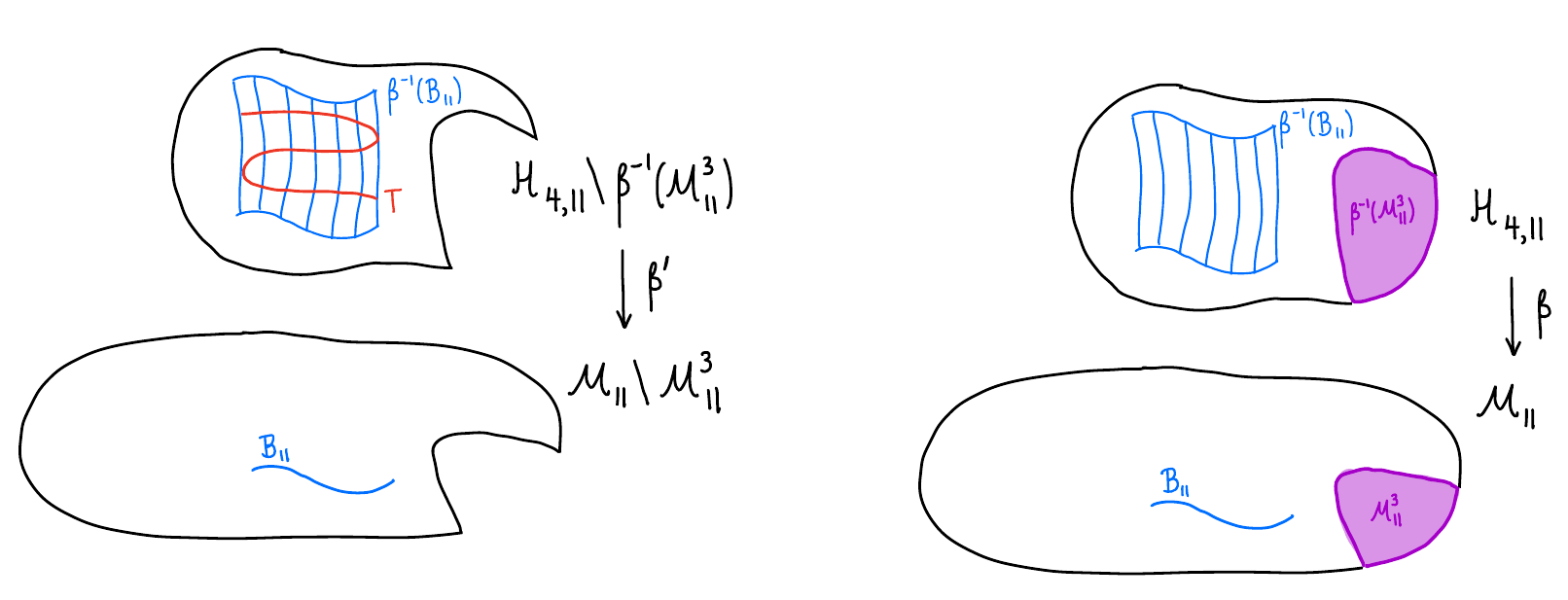}
    \caption{Modulo classes supported on $\M_{11}^3$, we realize $[\B_{11}]$ as a non-zero multiple of $\beta'_*[T]$, which is tautological.}
    \label{sketch}
\end{figure}

We shall do this by finding a tautological class $T \in 
A^*(\H_{4,11} \smallsetminus \beta^{-1}(\M_{11}^3))$ whose push forward along $\beta'$ is a non-zero multiple of $[\B_{11}] \in A^*(\M_{11} \smallsetminus \M_{11}^3)$ (see Figure \ref{sketch}).
Our class $T$ will be supported on
the preimage $\iota: \beta^{-1}(\B_{11}) \subset \H_{4,11}$, that is $T = \iota_*t$ for some $t \in A^*(\beta^{-1}(\B_{11}))$.
For $g \geq 6$, any degree $2$ map from a genus $g$ curve to an elliptic curve $C \to E$ is unique by the Castelnuovo--Severi inequality. 
Also by the Castelnuovo--Severi inequality, the preimage $\beta^{-1}(\B_{g})$ consists of degree $4$ covers of the form $C \to E \to \pp^1$ (the unique degree $2$ cover $C \to E$, followed by any degree $2$ cover $E \to \pp^1$). The fibers $\beta^{-1}(\B_g) \to \B_g$ are therefore $1$-dimensional (identified with $\Pic^2 E$). Hence, our $t$ should have codimension $1$ in $A^*(\beta^{-1}(\B_{11}))$, and consequently $T = \iota_* t$ has codimension $5$ inside $A^*(\H_{4,11} \smallsetminus \beta^{-1}(\M_{11}^3))$.

In genus $10$, we found that $\beta^{-1}(\B_{10})$
may be interpreted as a degeneracy locus occurring in the ``expected codimension." We used this to show that the fundamental class of $\beta^{-1}(\B_{10})$ is tautological in $A^*(\H_{4,10} \smallsetminus \beta^{-1}(\M_{10}^3))$. We also proved that $A^*(\beta^{-1}(\B_{10}))$ is generated by restrictions of tautological classes on $\H_{4,10}$. Pushing forward, it follows that all classes supported on $\B_{10}$ are tautological.

In genus $11$, however, the analogous construction fails: 
descriptions of $\beta^{-1}(\B_{11}) \subset \H_{4,11}$ as a degeneracy locus occur in the wrong codimension. However, the \emph{expected class} of an appropriate degeneracy locus on $\H_{4,11}$ yields a tautological
class in $A^5(\H_{4,11} \smallsetminus \beta^{-1}(\M_{11}^3))$, which --- using excess intersection formulas --- is the push forward of
 a divisor $t$ supported on $\beta^{-1}(\B_{11})$.
The key step is then to \emph{show that $t$
 meets each fiber of $\beta^{-1}(\B_{11}) \to \B_{11}$ with non-zero degree}.
Then, even though it remains unknown if $[\beta^{-1}(\B_{11})]$ is tautological on $\H_{4,11}$, we manage to show that a non-zero multiple of $[\B_{11}]$, and hence $[\B_{11}]$ itself, is tautological in $A^*(\M_{11})$.

\subsection{Outline of the paper}
In Section \ref{dl}, we explain how to realize $\beta^{-1}(\B_{11}) \subset \H_{4,11}$ as a degeneracy locus where a map of vector bundles drops rank. The expected codimension of this degeneracy locus is $5$. The excess Porteous formula (discussed in Section \ref{exp}) determines a divisor $t \in A^1(\beta^{-1}(\B_{11}))$ such that $\iota_* t$ is tautological. Then, in Section \ref{degcalc}, we calculate the degree of $t$ along each fiber of $\beta^{-1}(\B_{11}) \to \B_{11}$ and show it is nonzero. This shows that $[\B_{11}]$ is tautological in $A^*(\M_{11})$. In the final Section \ref{fs}, we consider the stratum of ``special bielliptic covers" in $\H_{4,11}$. This stratum maps finitely and surjectively onto $\B_{11}$ and its Chow ring is generated by pullbacks of tautological classes. Thus, using the push-pull formula, we 
conclude that all classes supported on $\B_{11}$ are tautological.

\subsection*{Acknowledgments} We are grateful to our advisors, Elham Izadi and Ravi Vakil, respectively, for the many helpful conversations.

\section{A degeneracy locus on the Hurwitz space} \label{dl}

The Hurwitz space $\H_{4,g}$ parametrizes degree $4$ covers $C \to \pp^1$ up to automorphisms of the target $\pp^1$. It comes equipped with a universal diagram:
\begin{equation} \label{ud}
\begin{tikzcd}
\C \arrow{dr}[swap]{f} \arrow{rr}{\alpha} & & \P \arrow{dl}{\pi} \\
& \H_{4,g}
\end{tikzcd}
\end{equation}
where $\C \to \H_{4,g}$ is the universal curve, $\alpha$ is the universal degree $4$ cover and $\P \to \H_{4,g}$ is the universal $\pp^1$-bundle. Since we are working with rational coefficients, we can assume that $\P = \pp \V^\vee$ is the projectivization of a rank $2$ vector bundle $\V$ on $\H_{4,g}$ with trivial first Chern class (see \cite[Section 2.3]{part2}).

\begin{definition}
The \emph{tautological ring} $R^*(\H_{4,g}) \subseteq A^*(\H_{4,g})$ is the subring generated by classes of the form $f_*(c_1(\omega_f)^i \cdot \alpha^* c_1(\omega_\pi)^j)$.
\end{definition}

Let $\beta: \H_{4,g} \to \M_g$ be the natural map.
The classes $\beta^*\kappa_i$ lie in $R^*(\H_{4,g})$. However, in general, not all tautological classes on $\H_{4,g}$ are pullbacks of tautological classes from $\M_g$. Nevertheless, tautological classes on $\H_{4,g}$ push forward to tautological classes on $\M_g$ in the following sense. Let $\M_g^3$ be the locus of curves of gonality $\leq 3$. 
\begin{thm}[Theorem 1.7 of \cite{part2}] \label{p2th}
The push forward along $\beta': \H_{4,g} \smallsetminus \beta^{-1}(\M_g^3) \to \M_g \smallsetminus \M_g^3$ sends tautological classes to tautological classes.
\end{thm}

\subsection{The splitting type stratification} \label{sts}
We recall the splitting type stratification of the Hurwitz space, arising from the Casnati--Ekedahl structure theorem \cite{CE,part1}.
The Casnati--Ekedahl structure theorem gives rise to two important vector bundles on $\P$. As in \cite[Section 2.1]{789}, we define
\[\E := \coker(\O_{\P} \rightarrow \alpha_*\O_{\C})^\vee \qquad \text{and} \qquad \F := \ker(\Sym^2 \E \to \alpha_*\omega_{\alpha}^{\otimes 2}).\]
These bundles have ranks $3$ and $2$ respectively, and both are relative degree $g + 3$ on the fibers of $\P \to \H_{4,g}$. 
In \cite{CDC}, Casnati--Del Centina describe $\beta^{-1}(\B_g)$ for $g \geq 10$ as the locus where $\F$
has a summand of degree $4$. Note that, by the Riemann--Hurwitz formula, we have 
\[\dim \H_{4,g} =  2g +3 \qquad \qquad \dim \B_g = 2g-2 \qquad \qquad \dim \beta^{-1}(\B_g) = 2g-1, \]
so
$\beta^{-1}(\B_{g}) \subset \H_{4,g}$ is codimension $4$.

Considering the splitting types of $\E$ and $\F$ on the fibers of $\P \to \H_{4,g}$ gives rise to a stratification of $\H_{4,g}$. Fixing a splitting type for both $\E$ and $\F$, each stratum is irreducible of codimension given by \cite[equation (3.3)]{789}.
In
\cite[equations (4.3)--(4.9)]{789}, we gave several restrictions on the allowed pairs of splitting types.
In genus $11$, using these conditions, we see that there are only three strata where $\F$ has a summand of degree $4$ or less:
\begin{enumerate}
    \item $\E$ splits as $(1, 6, 7)$ and $\F$ splits as $(2, 12)$, occurring in codimension $2$. This is the preimage of the hyperelliptic locus (see \cite[equation (4.5)]{789}). In particular it is contained in $\beta^{-1}(\M_g^3)$.
    \item $\E$ splits as $(2,6,6)$ and $\F$ splits as $(4,10)$, occurring in codimension $4$. Casnati--Del Centina call these the ``general bielliptic covers."
    \item $\E$ splits as $(2, 5, 7)$ and $\F$ splits as $(4, 10)$, occurring in codimension $5$. Casnati--Del Centina call these the ``special bielliptic covers." The special bielliptic covers $C \xrightarrow{\varphi} E \xrightarrow{\sigma} \pp^1$, where $C$ is of genus $11$, are characterized by the property that the branch locus of $\varphi$ is linearly equivalent to $\sigma^*\O_{\pp^1}(5)$ on $E$.
\end{enumerate}

\begin{rem}
Notice that the special bielliptics are a divisor within $\beta^{-1}(\B_{11})$, and this stratum maps finitely onto $\B_{11}$. We are unsure if the fundamental class of stratum (3) is tautological, so it is not our desired class $t$ from Section \ref{overview}. However, once we know that $[\B_{11}]$ is tautological, pushing forward from this stratum will help us to show that all classes supported on $\B_{11}$ are also tautological.
\end{rem}

In summary, $\beta^{-1}(\B_{11})$ is the union of strata (2) and (3) above:
 \[\beta^{-1}(\B_{11}) = \{b \in \H_{4,11}: \F_{\pi^{-1}(b)} \cong \O(4) \oplus \O(10)\}.  \]
Furthermore, the only other stratum where $\F$ acquires a summand of degree $< 4$ is (1) above.
Hence, 
working in the complement of $\beta^{-1}(\M_{11}^3)$, we can describe $\beta^{-1}(\B_{11})$ as the locus where $\F$ has a summand of degree $\leq 4$. This in turn can readily be described as a determinantal degeneracy locus.

\subsection{Description as a degeneracy locus}
From now on, we will work on $\H^\circ := \H_{4,11} \smallsetminus \beta^{-1}(\M_{11}^3)$. Over $\H^\circ$, the restriction of $\F$ to fibers of $\pi$ always has summands of degree at least $4$. Hence, by cohomology and base change, $\pi_* \F(-5)$ and $\pi_* \F(-4)$ are vector bundles on $\H^\circ$.
Now, we may describe $\beta^{-1}(\B_{11}) \subset \H^\circ$ as the locus where the multiplication map
\begin{equation} \label{themap}\mu: \pi_* \F(-5) \otimes \pi_* \O_{\P}(1) \to \pi_*\F(-4) 
\end{equation}
fails to be surjective. Let us write $D = \beta^{-1}(\B_{11}) \subset \H^\circ$ for this degeneracy locus.

Now, 
$\F$ has rank $2$ and relative degree $14$ on the fibers of $\pi$, so $\F(-5)$ has relative degree $4$. On $\pp^1$, the Euler characteristic of a rank $2$, degree $4$ bundle is $6$. 
Similarly, $\F(-4)$ has relative degree $6$, and thus Euler characteristic $8$ on fibers of $\pi$.  Hence, the ranks of the bundles in \eqref{themap} are
\[\rank \pi_* \F(-5) \otimes \pi_*\O_{\P}(1) = 6 \cdot 2 = 12 \qquad \text{and} \qquad \rank \pi_* \F(-4) = 8. \]
Thus, the expected dimension of the locus where \eqref{themap} fails to be surjective  is
\[\rank \pi_* \F(-5) \otimes \pi_*\O_{\P}(1) - \rank \pi_* \F(-4) + 1 = 5\]
(see \eqref{expco} below).
However, in our case, $D$ is codimension $4$. This is exactly the kind of setting where the excess Porteous formula is useful.

\subsection{The excess Porteous formula} \label{exp}
The usual Porteous formula tells us the fundamental class of the locus where a map of vector bundles drops rank when it occurs in the expected codimension. If $\sigma: A \to B$ is a map of vector bundles of ranks $a$ and $b$ on $X$, then the locus $D_k(\sigma)$ where $\sigma$ has rank $k$ or less has expected codimension $(a - k)(b - k)$. As a special case, the locus $D_{b-1}(\sigma)$, where $\sigma$ fails to be surjective has expected codimension
\begin{equation} \label{expco}
a - b + 1. \end{equation}
When $D_k(\sigma)$ occurs in the expected codimension, its fundamental class is given by a universal formula in terms of the Chern classes of $A$ and $B$ (see \cite[Section 14.4]{fulton}):
\begin{equation} \label{usualp}[D_k(\sigma)] = \Delta^{a-k}_{b-k}(c(B)/c(A)) \in A^{(a-k)(b-k)}(X). 
\end{equation}
Above, $c(B)/c(A)$ means we formally expand the ratio of total Chern classes; then given a class $c = 1 + c_1 + c_2 + \ldots$ where $c_i$ is the codimension $i$ component,
$\Delta^p_q(c)$ denotes the determinant of the $p \times p$ matrix $(c_{q + j - i})_{1 \leq i, j \leq p}$. For us, the precise formula is not so important. It is enough just to know that $\Delta^{a-k}_{b-k}(c(B)/c(A))$ is some polynomial in the Chern classes of $A$ and $B$.

Now we consider the case when $D = D_k(\sigma)$ does not necessarily have the expected codimension. Assume that $D_{k-1}(\sigma) = \varnothing$. 
Then, on $D_k(\sigma)$, there are kernel and cokernel bundles
\[0 \rightarrow K \rightarrow A|_D \to B|_D \to C \rightarrow 0.\]
We have $\rank K = a - k$ and $\rank C = b - k$.
Suppose further that $\iota: D \subset X$ is a local complete intersection of codimension $d$ with normal bundle $\N_{D/X}$. 
Let $p = (a - k)(b - k) - d$ be the difference of the expected and actual codimensions.
Then, by \cite[Example 14.4.7]{fulton}, we have
\begin{equation} \label{tada} 
\iota_* \left(\left[c(K^\vee \otimes C)/c(\N_{D/X})\right]_p\right) = \Delta^{a-k}_{b-k}(c(B)/c(A))
\end{equation}
where on the left, we formally expand the ratio of total Chern classes and then take the codimension $p$ piece.
This expression is called the \emph{excess Porteous formula}.
When $d = (a- k)(b-k)$ is the expected codimension, the left hand side is $\iota_* [1]$ which is just the fundamental class of $D$ in $X$. Hence, \eqref{tada} generalizes the usual Porteous formula \eqref{usualp}.

\subsection{Application to our situation}
We would like to apply the excess Porteous formula to $D = \beta^{-1}(\B_{11})$. In the notation of the previous subsection, we have $D = D_7(\mu)$, where $\mu$ is the multiplication map of \eqref{themap}. First we must check the following:
\begin{lem}
We have $D_6(\mu) = \varnothing$.
\end{lem}
\begin{proof}
We know $\F$ has splitting type $(4, 10)$ on all fibers of $\P$ over points in $D$. By cohomology and base change, we are reduced to checking that the multiplication map
$H^0(\pp^1, \O(-1) \oplus \O(5)) \otimes H^0(\pp^1, \O(1)) \to H^0(\pp^1, \O \oplus \O(6))$
has cokernel of rank $1$.
\end{proof}

Now let $K$ and $C$ be the kernel and cokernel bundles on $D$, which sit in an exact sequence
\begin{equation} \label{kc} 0 \rightarrow K \rightarrow [\pi_* \F(-5)\otimes \pi_*\O_{\P}(1)]|_D \rightarrow [\pi_* \F(-5)]|_D \rightarrow C \rightarrow 0. \end{equation}
Their ranks are $\rank K = 5$ and $\rank C = 1$.
The expected codimension of $D$ is $5$, but its actual codimension is $4$. Thus, the excess Porteous formula \eqref{tada} tells us that
\begin{equation} \label{oursituation}\iota_*
\left(\left[c(K^\vee \otimes C)/c(\N_{D/\H^\circ})\right]_1\right) = \Delta^{5}_{1}\left(\frac{c(\pi_*\F(-4))}{c(\pi_*\F(-5) \otimes \pi_*\O_{\P}(1))}\right).
\end{equation}
Let 
\[t := c_1(K^\vee \otimes C) - c_1(\N_{D/\H^\circ}) \in A^1(D)\] 
so that the left-hand side of \eqref{oursituation} is $\iota_* t$.

\begin{lem} \label{itist}
We have $\iota_*t$ is tautological in $A^5(\H^\circ)$.
\end{lem}
\begin{proof}
By \eqref{oursituation}, it will suffice to show that the Chern classes of $\pi_*\F(-4), \pi_*\F(-5)$ and $\pi_* \O_{\P}(1)$ are tautological. The bundle $\pi_* \O_{\P}(1)$ is the universal rank $2$ bundle pulled back from $\BSL_2$, so its first Chern class vanishes. Let $c_2 = c_2(\pi_* \O_{\P}(1))$ and
let $z = c_1(\O_{\P}(1))$.
By the projective bundle theorem, $z^2 = -\pi^* c_2$ and we can write
 $c_i(\F) = \pi^* b_i + \pi^*b_i'z$ for $b_i \in A^i(\H^\circ)$ and $b_i' \in A^{i-1}(\H^\circ)$.
 The classes $b_i, b_i'$ and $c_2$ are examples of the \emph{Casnati--Ekedahl classes} defined in \cite[Definition 3.9]{part1}. These classes were shown to be tautological in \cite[Theorem 3.10]{part1}. Finally, the Grothendieck--Riemann--Roch formula tells us how to express the Chern classes of $\pi_*\F(-4)$ and $\pi_*\F(-5)$ as polynomials in $b_i, b_i'$ and $c_2$ in $A^*(\H^\circ)$. Hence, they are tautological.
\end{proof}

Combining Lemma \ref{itist} with
Theorem \ref{p2th} now gives:

\begin{cor} \label{goodone}
We have $\beta'_* \iota_* t$ is is tautological in $\M_{11} \smallsetminus \M_{11}^3$.
\end{cor}

Our next goal is to 
show that $t$ maps with non-zero degree onto $\B_{11}$ so that $\beta'_*\iota_* t$ is a \emph{non-zero} multiple of $[\B_{11}]$.

\section{The degree of $t$ on fibers} \label{degcalc}
In this section, we calculate the degree of $t$ along a fixed fiber of $D = \beta^{-1}(\B_{11}) \to \B_{11}$. Each fiber $Z$ is an elliptic curve, and we must determine the degrees of the bundles $K, C$ and $\N_{D/\H^\circ}$ restricted to $Z \subset D$.

\subsection{Notation}
Let $C_0$ be a fixed genus $11$ bielliptic curve and let $\varphi: C_0 \to E$ be the unique degree $2$ map to an elliptic curve. Set $S := \varphi_* \O_C$, which is a rank $2$ vector bundle on $E$. We have
\[\deg(S) = \chi(E, S) = \chi(C_0, \O_{C_0}) = 1 - g(C_0) = -10.\]
In addition, let $Z := \Pic^2 E$ so that $Z$ is the fiber of $D = \beta^{-1}(\B_{11}) \to \B_{11}$ over $[C_0] \in \B_{11}$.

We write $\nu: E \times Z \to \Pic^2 E = Z$ for the (vertical) projection map, and $p: E \times Z \to E$ for the (horizontal) projection map.
Choose an isomorphism $Z \cong E$ and let $\Delta$ be the class of the diagonal on $E \times Z$. 
Let $v = \nu^* \mathrm{pt}$ be the class of a vertical fiber, and $h = p^* \mathrm{pt}$ be the class of a horizontal fiber.
\begin{center}
\includegraphics[width=3in]{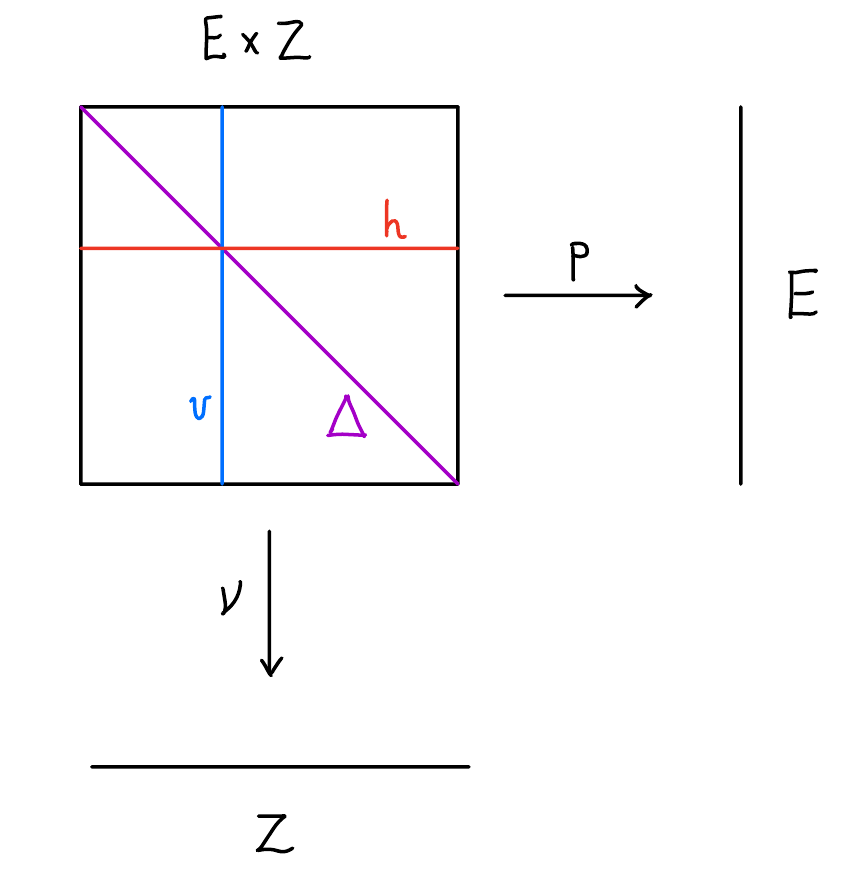}
\end{center}
The intersection numbers are
\[h^2 = v^2 = \Delta^2 = 0 \qquad h\cdot \Delta = h \cdot v = v \cdot \Delta = 1.\]
To prove that $\Delta^2 = 0$, we use adjunction for $\Delta$ on the surface $E \times Z$. Since $K_E = K_Z = \O$, we have $K_{E \times Z} = \O$. Therefore
\[0 = 2g(\Delta) - 2 = (K_{E \times Z} + \Delta) \cdot \Delta = K_{E \times Z} \cdot \Delta + \Delta^2 = \Delta^2.\]

Next, let $\L := \O_{E \times Z}(\Delta + h + sv)$ be a Poincar\'e line bundle on $E \times Z$. Our final degree calculation should not depend on the choice of $s$.
Define $V := \nu_* \L$, which is a rank $2$ vector bundle on $Z$.
The natural map 
\[\nu^* V = \nu^*\nu_* \L \to \L\]
on $E \times Z$
induces a map $\gamma: E \times Z \to \pp V^\vee =: \pp$ such that $\gamma^* \O_{\pp V^\vee}(1) \cong \L$. We think of $\gamma$ as the ``universal map of $E$ to a $\pp^1$." 

Our family of degree $4$ covers over $Z$ is $C_0 \times Z \xrightarrow{\varphi \times 1} E \times Z \xrightarrow{\gamma} \pp V^\vee.$  This induces a map $Z \to \H_{4,11}$ so that the outer triangle below is the base change of the universal diagram \eqref{ud} along $Z \to \H_{4,11}$.
\begin{equation} \label{thefam}
\begin{tikzcd}
C_0\times Z \arrow[r, "\varphi\times 1"] \arrow[rd, "f_0"'] \arrow[rr, "\alpha_0", bend left=49] & E\times Z \arrow[d, "\nu"] \arrow[r, "\gamma"] & \pp V^{\vee} \arrow[ld, "\pi_0"] \\
& Z & 
\end{tikzcd}
\end{equation}

By Grothendieck--Riemann--Roch (along a projection with elliptic fibers so $\td_{\nu} = 1$), we have
\[c_1(V) = c_1(\nu_*\L) = \nu_* \ch_2(\L) = \frac{1}{2} c_1(\L)^2 = \frac{1}{2}(\Delta + h + sv)^2 = \frac{1}{2}(2 + 4s) = 1 + 2s.\]
Above, and in what follows, we often omit writing $\deg$ in front of a first Chern class of a vector bundle on a curve, so
$c_1(V) = 1 + 2s$ is understood to mean $\deg c_1(V) = 1 + 2s$.
The rank $2$ bundle $V$ on $Z$ is the restriction of the universal bundle $\V$ from $\H_{4,11}$. Recall that we are assuming in our moduli problem for $\H_{4,g}$ that $c_2(\V)=0$ (see \eqref{ud} and the following discussion). This assumption corresponds to choosing $s = -\frac{1}{2}$. (We could just set $s = -\frac{1}{2}$ now, but we will keep $s$ around to confirm that it cancels out of the final answer.)

On $\pp := \pp V^\vee$, we have a tautological sequence
\begin{equation} \label{tauts} 0 \rightarrow \O_{\pp}(-1) \rightarrow V^\vee \rightarrow V' \rightarrow 0, 
\end{equation}
wherein we see $c_1(V') = c_1(\O_{\pp}(1)) - c_1(V)$. Hence, the relative tangent bundle of $\pp$ over $Z$ has
\[c_1(T_{\pp/Z}) = c_1(\O_{\pp}(1)) + c_1(V') = 2c_1(\O_{\pp}(1)) - c_1(V).\]
In particular, if $c_1(\L) = \Delta + h + sv$, then we see
\begin{equation} \label{cg} c_1(\gamma^*T_{\pp/Z}) = 2c_1(\L) - \nu^*c_1(V) = 2(\Delta + h + sv) - (1 + 2s)v = 2\Delta + 2h - v,\end{equation}
which is independent of $s$, as it must be.

\subsection{The normal bundle to $D \subset \H^\circ$.}
To find the degree of $\N_{D/\H^\circ}$ along $Z$ first note that we have an exact sequence
\[0 \rightarrow \N_{Z/D} \rightarrow \N_{Z/\H^\circ} \rightarrow \N_{D/\H^\circ}|_Z \rightarrow 0. \]
Because $Z$ is a fiber, we have $\deg(\N_{Z/D}) = 0$. Hence, $\deg(\N_{D/\H^\circ}|_Z) = \deg(\N_{Z/\H^\circ})$.
Meanwhile, from the sequence
\[0 \rightarrow \T_Z \rightarrow \T_{\H^\circ}|_Z \rightarrow \N_{Z/\H^\circ} \rightarrow 0,\]
we see $\deg(\N_{Z/\H^\circ}) = \deg(\T_{\H}|_Z)$. Hence, \[\deg(\N_{D/\H^\circ}|_Z) = \deg(\T_{\H}|_Z), \]
where we use that because $Z \subset \H^\circ$, $\T_{\H}|_Z=\T_{\H^\circ}|_{Z}$.

\subsubsection{Tangent space to $\H$} \label{tsh}
On the universal curve $\C$ over $\H = \H_{4,g}$ as in \eqref{ud}, there is a short exact sequence
\[0 \rightarrow \T_{\C/\H} \rightarrow \alpha^* \T_{\P/\H} \rightarrow \mathcal{R} \rightarrow 0,\]
where $\mathcal{R}$ is supported on the universal ramification divisor. First push forward by $\alpha$, which is exact because $\alpha$ is finite. Then push forward by $\pi$ to get the long exact sequence on $\H$:
\begin{equation} \label{les}0 \rightarrow f_* \T_{\C/\H} \rightarrow \pi_*((\alpha_* \O_\C)\otimes \T_{\P/\H}) \rightarrow f_* \mathcal{R} \rightarrow R^1f_* \T_{\C/\H} \rightarrow R^1\pi_*((\alpha_*\O_C) \otimes \T_{\P/\H}) \rightarrow 0. 
\end{equation}
Since the degree of the tangent bundle on a genus $g \geq 2$ curve is $2 - 2g < 0$, the first term above is zero. We also identify $R^1f_* \T_{\C/\H}$ as the pullback to $\H$ of the tangent bundle of $\M_g$. To see this, note that the tangent bundle to $\M_g$ is $(p_*(\Omega_p^{\otimes 2})^{\vee}$, where $p$ is the universal curve over $\M_g$, and $\Omega_p$ is the relative cotangent bundle by \cite[p.344]{ACG}. By Serre duality, we can identify this bundle with $R^1p_*(\T_p)$. Because $\C\rightarrow \H$ is a family of curves, it is pulled back from the universal curve over $\Mg$, so $R^1f_* \T_{\C/\H}=\beta^*R^1p_*(\T_p)=\beta^* \T_{\Mg}$.

To help understand the other terms, consider the exact sequence
\[0 \rightarrow \O_{\P} \rightarrow \alpha_* \O_{\C} \rightarrow \E^\vee \rightarrow 0.\]
on $\P$.
Tensoring with $\T_{\P/\H}$ and pushing forward by $\pi$ gives
\begin{equation} \label{f2} 0 \rightarrow \pi_* \T_{\P/\H} \rightarrow \pi_*((\alpha_*\O_C) \otimes \T_{\P/\H}) \rightarrow \pi_*(\E^\vee \otimes \T_{\P/\H}) \rightarrow 0. 
\end{equation}
The subbundle $\pi_*\T_{\P/\H} \subset \pi_*((\alpha_*\O_C) \otimes \T_{\P/\H})$
corresponds to the first order deformations of maps $C \to \pp^1$ induced by automorphisms the target $\pp^1$. The tangent space to the Hurwitz space is first order deformations of the map modulo those induced by such automorphisms. Quotienting this out of the second and third nonzero terms of \eqref{les}, we obtain the exact sequence
\begin{equation} \label{tseq} 0 \rightarrow \pi_*(\E^\vee \otimes \T_{\P/\H}) \rightarrow \T_{\H} \rightarrow \beta^* \T_{\M_g} \rightarrow R^1\pi_*((\alpha_*\O_C) \otimes \T_{\P/\H}) \rightarrow 0.
\end{equation}
The first term above tells us how to find the vertical tangent bundle along a fiber in terms of the bundle $\E$.

Consider our family $\alpha_0: C_0 \times Z \to \pp$ as in \eqref{thefam}. Define 
\[\E_0 = \coker(\O_{\pp} \to (\alpha_0)_* \O_{C_0 \times Z})^\vee,\]
so $\E_0$ is the restriction of the universal $\E$ on $\P$ to $\pi^{-1}(Z)$. 
Recall that on the fibers of $\pp \to Z$, our $\E_0$ always has a degree $2$ summand (since $Z \subset D$, and the latter is the union of the strata (2) and (3) from Section \ref{sts}). This degree $2$ summand gives rise to a distinguished quotient
\[\E_0 \to (\pi_0^*L)(2) \]
where $L$ is a line bundle on $Z$. In fact, twisting, taking duals and pushing forward, we have that $L = (\pi_{0*}(\E^\vee_0 \otimes \O_{\pp}(2)))^\vee$. Here, $L$ is what we called an HN bundle for $\E$ restricted to some stratum in \cite[Section 3]{789}.
By the previous paragraph, the tangent bundle to the vertical fiber $Z$ is \[\pi_{0*}(\E_0^\vee \otimes T_{\pp/Z}) = \det V^\vee \otimes \pi_{0*}(\E^\vee_0 \otimes \O_{\pp}(2)) = \det V^\vee \otimes L^\vee.\]
On the other hand, $Z$ is an elliptic curve, so its tangent bundle is trivial. Hence, we learn
\begin{equation} \label{leq}c_1(L) = -c_1(V) = -(1+2s),\end{equation}
which will be useful later.

We want to find the degree of $\T_\H|_{Z}$. We have $\deg(\beta^*\T_{\M_g}|_Z) = 0$ because $Z$ is a fiber of $\beta$. 
As just discussed, $\pi_{0*}(\E_0^\vee \otimes \T_{\pp/Z})$ has degree $0$.
Hence, by \eqref{tseq}, we have
\begin{align} \deg \T_{\H}|_Z &=  - \deg(R^1 \pi_{0*}((\alpha_{0*} \O_{C_0 \times Z}) \otimes T_{\pp/Z})) \notag
\intertext{Next, we claim that $\deg(\pi_{0*}(\alpha_{0*}\O_{C_0 \times Z} \otimes \T_{\pp/Z})) $ is also $0$. As just discussed, the right-hand term of the filtration \eqref{f2} has degree $0$ since it is the tangent bundle to $Z$. The left-hand term of the filtration is $\pi_{0*}\T_{\pp/Z} = \pi_*(\O_{\pp}(2) \otimes \det V^\vee) = \Sym^2 V \otimes \det V^\vee$,
which also seen to be degree $0$ by the splitting principle. Hence, the previous line equals}
&= \deg(\pi_{0*}((\alpha_{0*} \O_{C_0 \times Z}) \otimes T_{\pp/Z})) - \deg(R^1 \pi_{0*}((\alpha_{0*} \O_{C_0 \times Z}) \otimes T_{\pp/Z})). \notag \\
\intertext{Note that $\alpha_{0*}\O_{\C_0 \times Z} = \gamma_*(\varphi \times 1)_* \O_{C_0 \times Z} = \gamma_*p^*(\varphi_*\O_{C_0}) = \gamma_*p^*S$, so the above becomes}
&= \deg(\pi_{0*} (\gamma_*p^*S \otimes T_{\pp/Z})) - \deg(R^1\pi_{0*}(\gamma_* (p^*S \otimes  T_{\pp/Z}))),  \notag \\
\intertext{or using the push-pull formula:}
&= \deg(\pi_{0*} (\gamma_*p^*S \otimes \gamma^*T_{\pp/Z})) - \deg(R^1\pi_{0*}(\gamma_* (p^*S \otimes \gamma^* T_{\pp/Z}))). \notag \\
\intertext{Because $\gamma$ is finite, $\gamma_*$ is exact, so $R^1\pi_{0*}\gamma_* = R^1(\pi_0 \circ \gamma)_* = R^1\nu_*$, and we get:}
&= \deg(\nu_*(p^*S \otimes \gamma^* T_{\pp/Z}) - \deg(R^1\nu_* (p^*S \otimes \gamma^* T_{\pp/Z})). \notag \\
\intertext{Finally, using Grothendieck--Riemann--Roch, we have}
&= [\nu_*(\ch(p^*S \otimes \gamma^* T_{\pp}) \cdot \mathrm{td}_\nu)]_1. \label{ha}
\end{align}
Since $\nu: E \times Z \to Z$ has trivial relative tangent bundle, we have $\mathrm{td}_{\nu} = 1$.
Next, $p^*S$ is a vector bundle pulled back from $E$, so 
\begin{equation} \label{chps} \ch(p^*S) = p^*\ch(S) = p^*(\rank(S) + c_1(S)) = p^*(\rank(S) + \deg(S) \cdot [\mathrm{pt}]) = 2 - 10h.
\end{equation}
Additionally, using \eqref{cg}, we have
\[\ch(\gamma^* T_{\pp/Z}) = 1 + (2\Delta + 2h - v) + \tfrac{1}{2}(2\Delta + 2h - v)^2.\]
Putting these together, we find
\[\ch(p^*S \otimes \gamma^* T_{\pp}) = \ch(p^*S) \cdot \ch(\gamma^* T_{\pp}) = (2 -10h)\cdot (1 + (2\Delta + 2h - v) + \tfrac{1}{2}(2\Delta + 2h - v)^2).\]
We want the degree $2$ piece (which pushes forward to the degree $1$ piece needed in \eqref{ha}):
\begin{align*}
\deg \ch_2(S \otimes \gamma^* T_{\pp}) &=-10h \cdot (2\Delta + 2h - v) + (2\Delta + 2h - v)^2 \\
&= -10(2 + 0 - 1) + (8 - 4 - 4) \\
&= -10.
\end{align*}
Thus, we have found one of our desired terms
\begin{equation} \label{nterm}\deg(\N_{D/\H}|_Z) = \deg(\T_{\H}|_Z) = -10.
\end{equation}


\subsection{The kernel and cokernel bundles}
Recall that $\F$ splits as $\O(4) \oplus \O(10)$ on each fiber of $\pi$ over $Z \subset D$.  Let us write $\F_0$ for $\F|_{\pi^{-1}(Z)} = \F|_{\pp}$.
On $Z$, there are line bundles $M$ and $N$ such that $\F_0$ admits a filtration
\begin{equation} \label{filt} 0 \rightarrow (\pi_0^*N)(10) \rightarrow \F_{0} \rightarrow (\pi_0^*M)(4) \rightarrow 0,\end{equation}
which restricts to the Harder--Narasimhan filtration on each fiber of $\pp \to Z$. These are what we called the HN bundles for $\F$ in \cite[Section 3]{789}.

Twisting \eqref{filt} by $\O_{\pp}(-5)$ and pushing forward, we see that $\pi_{0*} \F_0(-5) = N \otimes \pi_{0*}\O_{\pp}(5)$.
Similarly, twisting by $\O_{\pp}(-4)$ and pushing forward, we get the filtration of $\pi_{0*}\F_0(-4)$ given in the right vertical column the diagram below. 
\begin{equation} \label{md}
\begin{tikzcd}
& & 0 \arrow{d} & 0 \arrow{d} \\
0 \arrow{r} & N \otimes \det V \otimes \pi_{0*}\O_{\pp}(4) \arrow{r} &
N \otimes \pi_{0*}\O_{\pp}(5) \otimes \pi_{0*}\O_{\pp}(1) \arrow{r} \arrow{d} & N \otimes \O_{\pp}(6) \arrow{d} \arrow{r} & 0 \\
& & \pi_{0*} \F_0(-5) \otimes \pi_{0*}\O_{\pp}(1) \arrow{r}{\mu|_Z} \arrow{d} & \pi_{0*} \F_0(-4) \arrow{d}  \\
&  & 0 &M \arrow{d} \\
& & & 0
\end{tikzcd}
\end{equation}
The top row above comes from tensoring $N$ with the usual multiplication map, which in turn is obtained by taking the dual of \eqref{tauts}, noting that $V'^\vee = \det V \otimes \O_{\pp}(-1)$, tensoring with $\O_{\pp}(5)$ and pushing forward.
From \eqref{md}, it is clear that the kernel bundle of $\mu|_Z$ is $K|_Z = N \otimes \det V \otimes \pi_{0*}\O_{\pp}(4)$ and the cokernel is $C|_Z = M$. Note that $\pi_{0*}\O_{\pp}(4) = \Sym^4 V$. Thus, by the splitting principle, we find 
\[ c_1(K|_Z) = \rank (\Sym^4 V)(c_1(N) + c_1(V)) + c_1(\Sym^4 V) = 5c_1(N) + 15c_1(V).\]
Using the splitting principle again, we have
\begin{align} \label{ckterm}
c_1(K|_Z^\vee \otimes C|_Z) &= (\rank K|_Z^\vee)c_1(C|_Z) + (\rank C|_Z) c_1(K|_Z^\vee) \notag \\
&= 5c_1(C|_Z) - c_1(K|_Z)  \\
&= 5c_1(M) - 5c_1(N) - 15c_1(V). \notag
\end{align}

\subsubsection{The degree of the HN bundle $M$}
The line bundle $M$ is closely related to the line bundle $L$ defined in Section \ref{tsh}. 
Just as in the proof of \cite[Lemma 4.3(1)]{789},
certain geometric considerations give rise to a non-vanishing section of $M^\vee \cong L^{\otimes 2}$, as we now explain. As in \cite[Section 3.1]{789}, our family of degree $4$ covers over $Z$ determines a global section of $\F_0^\vee \otimes \Sym^2 \E_0$, which tells us the pencil of quadrics that cut out our curve inside $\pp \E_0^\vee$.

The canonical quotient $\E_0 \to (\pi_0^*L)(2)$ induces a quotient $\F_0^\vee \otimes \Sym^2 \E_0 \to \F_0^\vee \otimes (\pi_0^*L^{\otimes 2})(4)$. Geometrically, this quotient corresponds to restricting our quadrics to the distinguished section $\pp((\pi_0^*L)(2)^\vee) \subset \pp \E_0^\vee$. Because
our source curves are irreducible, the quadrics cannot vanish identically on the distinguished section in any fiber over $Z$.
Tensoring the dual of \eqref{filt} with $(\pi_0^*L^{\otimes 2})(4)$ and applying $\pi_{0*}$, we see that this gives rise to a non-vanishing section of $M^\vee \otimes L^{\otimes 2}$.
Thus, $M \cong L^{\otimes 2}$, so using \eqref{leq}, we have
\begin{equation} \label{mdeg} c_1(M) = 2c_1(L) = -2(1+2s).
\end{equation}

\subsubsection{The degree of the HN bundle $N$}
Taking the determinant of \eqref{filt}, we have
\[N \cong M^\vee \otimes \pi_{0*}\det \F_0(-14) \cong M^\vee \otimes \pi_{0*}\det \E_0(-14),\]
where the second isomorphism above comes from the canonical isomorphism $\det \F_0 \cong \det \E_0$ in the structure theorems for degree $4$ covers (see \cite{CE} or \cite[Section 3.2]{part1}).

Let us write $z := c_1(\O_{\pp}(1))$.
On $\pp = \pp V^\vee$, the projective bundle theorem gives $z^2 = \pi_0^*c_1(V) \cdot z \in A^2(\pp)$  (note that $c_2(V) = 0$ since $Z$ is a curve).
Let $\overline{v} = \pi_0^*[\mathrm{pt}]$. By \eqref{thefam}, we have $\gamma^*\overline{v} = v$.
On $\pp$, we have $\deg(\overline{v} \cdot z) = 1$, so
\[\deg(z^2) = \deg(\pi_0^*c_1(V) \cdot z) = (1+2s)\deg(\overline{v} \cdot z) = 1 + 2s.\]

Next, we want to find $c_1(\E_0) = -c_1(\gamma_* p^*S)$. For this, we use Grothendieck--Riemann--Roch for $\gamma$. The first ingredient is to compute the relative Todd class.  
We know the tangent bundle of $E \times Z$ is trivial, so the relative Todd class is
\[\td_{\gamma} = \frac{\td_{E \times Z}}{\gamma^* \td_{\pp}} = \frac{1}{\gamma^* \td_{\pp}} = 1 - \frac{1}{2}\gamma^*c_1(T_{\pp}) + \frac{1}{6}\gamma^*c_1(T_{\pp})^2. \]
By \eqref{chps}, we have $\ch(p^*S) = 2 - 10h$, so Grothendieck--Riemann--Roch gives
\begin{align*}
\ch(\gamma_*p^*S) &= \gamma_*\left[\ch(p^*S) \cdot \mathrm{td}_\gamma\right] \\
&= \gamma_*\left[(2 -10h) \cdot \left(1 - \frac{1}{2}\gamma^*c_1(T_{\pp}) + \frac{1}{6} \gamma^*c_1(T_{\pp})^2\right) \right] \\
&= \gamma_*\left[ 2 - \gamma^*c_1(T_{\pp}) -10h + \ldots\right]. \\
\intertext{Since $\gamma$ has degree $2$, we have $\gamma_*[E \times Z] = 2[\pp]$ and the above becomes}
&=4[\pp]  -\gamma_*\gamma^*c_1(T_{\pp}) -10 (\gamma_*h) \\
&= 4[\pp] - 2c_1(T_{\pp}) - 10 (\gamma_* h).
\end{align*}
We have
\[c_1(T_{\pp}) = c_1(T_Z) + c_1(T_{\pp/Z})= 0 + 2z - \pi_0^*c_1(V) = 2z - (1+2s)\overline{v}.\]

Next, we need to find $\gamma_*h$. This class has the form $\gamma_* h = az + b \overline{v}$. To find the coefficients $a, b$, intersect with $z$ and $\overline{v}$ and use the push-pull formula:
\[a = (\gamma_*h)\cdot \overline{v} = h \cdot \gamma^*\overline{v} = h \cdot v = 1\]
and
\[(1+2s)a + b = (\gamma_*h)\cdot z = h \cdot \gamma^*z = h \cdot (\Delta + h + sv) = 1 + s,\]
from which we see $b = -s$. Putting this together, we have
\begin{align*}
c_1(\gamma_*S) = -2c_1(T_{\pp}) - 10(\gamma_*h) = -2\cdot (2z - (1+2s)\overline{v}) -10(z - s \overline{v}) = -14z + (2 + 14s) \overline{v}.
\end{align*}
It follows that
\[c_1(\det \E_0(-14)) = c_1(\E_0) -14z = -c_1(S) - 14z = -(2+14s)\overline{v}.
\]
Finally,
\begin{equation} \label{ndeg} c_1(N) = -c_1(M) + c_1(\pi_{0*}\det \E_0(-14)) = 
2(1+2s) -(2+14s) = -10s.
\end{equation}

\subsection{Combining all the terms} \label{last}
We now combine \eqref{nterm} and \eqref{ckterm} to
see that the degree of $t$ along the fiber $Z$ is
\begin{align*}
\deg(t \cdot Z) &= c_1(K|_Z^\vee \otimes C|_Z) - c_1(\N_{D/\H}|_Z) \\
&= 5c_1(M) - 5c_1(N) - 15c_1(V) - (-10).\\
\intertext{Substituting in \eqref{mdeg} and \eqref{ndeg}, this becomes}
&= 5\cdot (-2(1 + 2s)) - 5 \cdot (-10s) - 15 \cdot (1 + 2s) - (-10) \\
&=-10 -20s + 50s -15 -30s + 10 \\
&= -15.
\end{align*}
Notice that the $s$ terms cancel as promised. Most importantly, we see $\deg(t \cdot Z) \neq 0$.

\section{Proof of Theorem \ref{main}} \label{fs}
By our calculation in Section \ref{last}, we have
\[ \beta'_* \iota_* t = \deg (t \cdot Z) [\B_{11}] = -15[\B_{11}] \in A^{10}(\M_{11} \smallsetminus \M_{11}^3). \]
The left-hand side is tautological by Corollary \ref{goodone}. Hence, $[\B_{11}] \in A^{10}(\M_{11} \smallsetminus \M_{11}^3)$ is tautological. Moreover, all classes supported on $\M_{11}^3$ are tautological by \cite[Section 1.1(1)]{789}, so $[\B_{11}]$ is tautological in $A^{10}(\M_{11})$.

To see that all classes on $\B_{11}$ are tautological, let $\Sigma$ be the stratum (3) of special bielliptics from Section \ref{sts}. From Casnati--Del Centina's characterization of special bielliptics \cite{CDC}, we see that $\Sigma \to \B_{11}$ is finite and surjective. Hence, the push forward map $a_*: A^*(\Sigma) \to A^*(\M_{11})$ surjects onto the group of cycles supported on $\B_{11}$.
The push forward of the fundamental class $a_*[\Sigma]$ is a multiple of $[\B_{11}]$, which we have shown is tautological. By \cite[Lemma 4.3(2)]{789}, we have $A^*(\Sigma)$ is generated by $a^*\kappa_1$ and $a^*\kappa_2$. Using the push-pull formula, it follows that all classes supported on $\B_{11}$ are tautological. \qed

\bibliographystyle{amsplain}
\bibliography{refs}
\end{document}